\documentclass[11pt]{amsart}

\usepackage[utf8]{inputenc}

\usepackage[margin=1.2in, centering]{geometry}
\usepackage{amsmath,amssymb,latexsym,mathtools}

\usepackage{amsthm}
\newtheorem{thm}{Theorem}[section]
\newtheorem{lem}[thm]{Lemma}

\newtheorem{prop}[thm]{Proposition}

\numberwithin{equation}{section}

\theoremstyle{definition}

\newtheorem{defn}[thm]{Definition}
\newtheorem{ex}[thm]{Example}
\newtheorem{remark}[thm]{Remark}

\newtheorem*{acknowledgement}{Acknowledgements}

\usepackage{graphicx}
\DeclareGraphicsExtensions{.pdf,.png,.jpg}
\usepackage{subfig}

\usepackage[usenames,dvipsnames]{color}
\usepackage{chngcntr}
\counterwithin*{equation}{section}

\newcommand{\bE}{\mathbb{E}}
\newcommand{\bN}{\mathbb{N}}
\newcommand{\bC}{\mathbb{C}}

\newcommand{\bR}{\mathbb{R}}

\newcommand{\bZ}{\mathbb{Z}}
\newcommand{\bT}{\mathbb{T}}

\newcommand{\cP}{\mathcal{P}}

\newcommand{\norm}[1]{{\left\Vert#1\right\Vert}}

\newcommand{\diam}[1]{\mathrm{diam}(#1)}

\usepackage{hyperref}
\hypersetup{linktocpage}

\begin{document}

    \title{A variant of the $\Lambda (p)$-set problem in Orlicz spaces}
	\author{Donggeun Ryou}
	\address{Department of Mathematics, University of Rochester, Rochester, NY, USA}
	\email{dryou@ur.rochester.edu}
	\date{\today}
	\subjclass{43a46 (primary) 46b09, 46e30 (secondary)}

\begin{abstract} 
	We introduce $ \Lambda(\Phi) $-sets as generalizations of $ \Lambda(p) $-sets. These sets are defined in terms of Orlicz norms. We consider $\Lambda(\Phi)$-sets when the Matuszewska-Orlicz index of $ \Phi $ is larger than $ 2 $. When $S$ is a $\Lambda(\Phi)$-set, we establish an estimate of the size of $ S \cap [-N,N] $ where $ N \in \bN $. Next, we construct a $ \Lambda(\Phi_1)$-set which is not a $ \Lambda(\Phi_2)$-set for any $ \Phi_2 $ such that $ \sup_{u \geq 1} \Phi_2(u) / \Phi_1(u) = \infty $ by using a probabilistic method. With an additional assumption about a subset $E$ of $\bZ$, we can construct such a $\Lambda(\Phi_1)$-set contained in $E$. These statements extend known results on the structure of $ \Lambda(p) $-sets to $\Lambda(\Phi)$-sets.
\end{abstract}
\maketitle

\section{Introduction}
A subset $ S $ of $ \bZ $ is called a \textit{$ \Lambda(p) $-set} if 
\[ \norm{\sum_{n \in S} a_n e^{2\pi i n x}}_{L^p(\bT)}\leq C(p) \norm{\sum_{n \in S} a_n e^{2\pi i n x}}_{L^q(\bT)} \]
for some $ q $ such that $  0< q < p $ and a constant $ C(p) $ which depends on $ p $ where $\bT = \bR/2\pi \bZ$. We are interested in $\Lambda(p)$-sets when $p>2$ because $S$ is a $\Lambda(p)$ set for $p>2$ if and only if 
\begin{equation}\label{1defp>2}
		\norm{\sum_{n \in S} a_n e^{2\pi i n x}}_{L^p(\bT)}\leq C(p) \bigg(\sum_{n \in S}|a_n|^2\bigg)^{1/2}
\end{equation}
and we can use properties of Fourier series. The inequality \eqref{1defp>2} is a direct implication of \cite[Theorem 1.4]{Rud60}. When $p \leq 2 $, the behavior of $ \Lambda(p) $-sets is different, see \cite[204p]{Bour01}. In \cite{Rud60}, Rudin studied properties of $ \Lambda(p) $-sets and the following theorems are related to this paper.
\begin{thm}\cite[Theorem 3.5]{Rud60}
	Let us denote by $ |A| $ a number of elements in a set $ A $. If $p >2$ and $ S $ is a $ \Lambda(p) $-set, then
	\begin{equation}\label{1sizeS}
		|S \cap [-N,N]| \leq C(p) N^{2/p}.
	\end{equation}
\end{thm}
\begin{thm}\cite[Theorem 4.8]
	{Rud60}\label{1ThmRudex}
	Let $ n $ be an integer larger than $1$. There exists a $ \Lambda(2n) $-set which is not a $ \Lambda(2n+\epsilon) $-set for any $ \epsilon>0 $.
\end{thm}
In the proof of Theorem \ref{1ThmRudex}, Rudin constructed explicit examples while no such explicit example is known when $p$ is not an even integer. However, Bourgain \cite{Bour89} used the probabilistic method to show that there exists a $ \Lambda(p) $-set when $ p >2 $ which is not a $ \Lambda(p+\epsilon) $-set for any $ \epsilon >0 $. The key step of \cite{Bour89} was to prove that for all $N$, there exists a set $ S $ of size $ N^{2/p} $ in $ \{ 1, \cdots, N \} $ which satisfies \eqref{1defp>2} where $C(p)$ is independent of $N$. This is where the probabilistic argument was used. In \cite{Tal95}, Talagrand generalized this key step of \cite{Bour89} that for all $N$, there exists a set $ S $ of size $ N^{2/p} $ in $ \{ 1, \cdots, N \} $ such that 
\begin{equation}\label{1eqTal1}
	\norm{\sum_{n \in S} a_n e^{2\pi i n x}}_{L^{p,1}(\bT)} \leq C(p) \bigg(  \sum_{n \in S} |a_n|^2\bigg)^{1/2}
\end{equation}
where $ L^{p,1}(\bT) $ is a Lorentz space and $C(p)$ does not depend on $N$. In \cite{Tal14}, the reader will find a simpler proof of the result. There were some works which studied analogues of $\Lambda(p)$-sets in more general function spaces. Grinnell \cite{Grin99} defined $ \Lambda(p,q) $-sets in terms of Lorentz spaces as generalizations of $ \Lambda(p) $-sets and studied properties of $ \Lambda(p,q) $-sets. Grinnell also showed that there exists a $ \Lambda(p_1,q_1) $-set which is not a $ \Lambda(p,q) $-set for any $ p>p_1  $ when $ 2 < p_1 \leq q_1 $. Astashkin \cite{Ast14} studied the $ \Lambda(p) $-sets in a setting of rearrangement invariant spaces and showed that there exists a $ \Lambda(p) $-set when $ p >2 $ such that \eqref{1defp>2} does not hold when the $ L^p $ norm is replaced by an Orlicz norm where the corresponding Orlicz space is properly contained in the $ L^p $ space.\\
\indent In another direction, as a follow up paper of \cite{Bour89}, Bourgain \cite{Bour89_2} proved that if $ p>4 $, there exists a $ \Lambda(p) $-set $ S $ contained in a square set $ Q \coloneqq \{k^2 : k \in \bN \} $ such that $ S $ is not a $ \Lambda(p+\epsilon) $-set for any $ \epsilon >0 $. The important intermediate step of \cite{Bour89_2} is to prove that
 \begin{equation}\label{1estQ}
 	\norm{\sum_{n \in Q} a_n e^{2\pi i n x}}_{L^p(\bT)} \leq C(p) N^{1/2-2/p} \bigg(\sum_{n \in Q} |a_n|^{2} \bigg)^{1/2} \qquad \mathrm{if} \  p>4
 \end{equation}
and to combine with the argument in \cite{Bour89}.\\
\indent In this paper, we consider $\Lambda(\Phi)$-sets which are generalizations of $\Lambda(p)$-sets. In order to define $\Lambda(\Phi)$-sets, we need the following definitions and notations. A function $ \Phi: [0,\infty) \rightarrow [0,\infty ] $ is called a \textit{Young function} if $ \Phi $ is convex, $ \Phi(0)=0 $ and $ \lim\limits_{u \rightarrow \infty} \Phi(u) = +\infty $. The \textit{Orlicz norm} of a measurable function $ f $ on a measure space $ (\Omega, \mu) $ with respect to a Young function $ \Phi $ is defined by
\[ \norm{f}_{\Phi} = \inf \bigg\{ k>0 : \int_\Omega \Phi (|f| / k)d\mu \leq 1 \bigg\}. \]
The norm $ \norm{f}_\Phi $ is also called the \textit{Luxemberg norm} of $ f $. The Orlicz space $ L^{\Phi}(\Omega, \mu) $ is defined as a space of all measurable functions $ f $ on $  \Omega $ such that $ \norm{f}_{\Phi} < \infty $.\\ 
A function $ \Phi : [0,\infty) \rightarrow [0,\infty] $ is called a \textit{nice Young function} or an \textit{$ N $-function} if $ \Phi $ is a continuous Young function such that $ \Phi(u) =0 $ if and only if $ u=0 $ and satisfies  
\[ \lim_{u \rightarrow 0} \frac{\Phi(u)}{u} =0 \qquad \mathrm{and} \qquad \lim\limits_{u \rightarrow \infty} \frac{\Phi(u)}{u} =+\infty.
\]
We denote that $ \Phi \in \Delta_2 $ if there exists $K_1 >0$ and $u_1 \geq 0$ such that $ \Phi(2u) \leq K_1\Phi(u) $ for $ u\geq u_1 $ and $ \Phi \in \nabla_2 $ if there exists $K_2 >2$ and $u_2 \geq 0$ such that $ \Phi(2u) \geq K_2\Phi(u) $ for $u \geq u_2$.
We also define the following indices.
\[ M^a (t, \Phi) \coloneqq \sup_{u >0} \frac{\Phi(tu)}{\Phi(u)}\qquad  M^\infty(t,\Phi) \coloneqq \limsup_{u\rightarrow \infty} \frac{\Phi(tu)}{\Phi(u)} \]
\[ \alpha_{\Phi}^i \coloneqq \lim\limits_{t \rightarrow 0+} \frac{\log M_i(t, \Phi)}{\log t} \qquad \beta_{\Phi}^i \coloneqq \lim_{t \rightarrow \infty} \frac{\log M_i
	(t ,\Phi)}{\log
	 t} \]
where $ i = a  $ or $ i =\infty $. The indices $ \alpha_\Phi^i $ and $ \beta_{\Phi}^i $ are called \textit{Matuszewska-Orlicz indices}. This is a special case of Boyd indices of rearrangement-invariant function spaces, see \cite{Mal89} and \cite{BS88}. These indices imply that
\begin{equation*}
	C_1(\epsilon,\Phi)u^{\alpha_{\Phi}^\infty-\epsilon} \leq \Phi(u) \leq C_2(\epsilon,\Phi) u^{\beta_{\Phi}^\infty +\epsilon} 
\end{equation*}
if $ u \geq 1 $ where $ C_1(\epsilon,\Phi) $ and $ C_2(\epsilon,\Phi) $ are constants that depend on $ \epsilon $  and $ \Phi $. We establish this in Lemma \ref{2propM}. Also, note that $ \alpha_\Phi^i \leq \beta_\Phi^i $.\\
Throughout the paper, we consider $ \Omega = \bT $ and $ \mu $ is the Lebesgue measure so that $ d\mu \coloneqq dx $. We denote $ L^\Phi(\Omega, \mu) $ by $ L^{\Phi}(\bT) $. If $ \Phi(x) = x^p $, we denote $ \norm{f}_{\Phi} $ by $ \norm{f}_p $ and denote $ L^p(\Omega, \mu) $ by $ L^p(\bT) $. We denote by $ X \lesssim Y $ when $ X \leq CY $ for some constant $ C >0 $ where the constant $ C $ may depend on $ \Phi,\Phi_0, \Phi_1$, $ \Phi_2 $ and some parameters determined by $ \Phi, \Phi_0, \Phi_1$ and $ \Phi_2 $ which are $p$, $p_0$, $p_1$, $p+\epsilon$ and $p_0 + \epsilon_0$, but does not depend on other parameters, such as $ N $. We write $ X \approx Y $ to denote that $ X \gtrsim Y $ and $ X \lesssim Y $. If there is a constant $ C_\epsilon $ that depends  on $ \epsilon $ such that $ X(\epsilon) \leq C_\epsilon Y(\epsilon) $, we denote it by $ X(\epsilon) \lesssim_{\epsilon} Y(\epsilon) $. Now, we state the definition of $ \Lambda(\Phi) $-sets and main theorems.
\begin{defn}
	A subset $ S $ of $ \bZ $ is called a \textit{$ \Lambda(\Phi) $-set} if there exists a nice Young function $ \Phi_1 $ such that $ \beta_{\Phi_1}^\infty < \alpha_{\Phi}^\infty $ and 
	\begin{equation}\label{1defnLambda}
		\norm{\sum_{n\in S} a_n e^{2\pi i n x}}_{\Phi}  \lesssim \norm{\sum_{n\in S} a_n e^{2\pi i n x}}_{\Phi_1}
	\end{equation}
	for all $ \sum_{n \in S} a_n e^{2\pi i nx} $ where $ a_n \in \bC $.
\end{defn}
If $ \alpha_\Phi^\infty > 2 $, \eqref{1defnLambda} can be replaced by
\begin{equation}\label{1defnLambda2}
	\norm{\sum_{n\in S} a_n e^{2\pi i n x}}_{\Phi}  \lesssim \bigg( \sum_{n \in S}|a_n|^2 \bigg)^{1/2}.
\end{equation}
The equivalence of \eqref{1defnLambda} and \eqref{1defnLambda2} follows from Theorem \ref{thmRud} in section \ref{sec:3}.
\begin{thm}\label{M1}
	Let $ \Phi $ is a nice Young function such that $ 2 < \alpha_\Phi^\infty $. If $ S $ is a $ \Lambda(\Phi) $-set, then
	\begin{equation}\label{1setsize}
		|S \cap [-N,N]| \lesssim (\Phi^{-1}(N))^2.
	\end{equation}
\end{thm}
\begin{thm}\label{M2}
	Assume that $ \Phi_0, \Phi $ are nice Young functions such that $ \Phi_0, \Phi \in \Delta_2 $ and $ 2\leq \alpha_{\Phi_0}^\infty $, $ \beta_{\Phi_0}^\infty < \alpha_{\Phi}^\infty $. Let $E$ be a subset of $\bZ$ and $\diam{E} \coloneqq |\max E - \min E|+1 $.\\
	For
	\[ \delta = \frac{(\Phi^{-1}(\diam{E}))^2}{|E|}, \]
	we consider independent and identically distributed random variables $ ( \xi_i )_{i \in \bZ} $ such that
    \[ P(\xi_i=1) = \delta \qquad P(\xi_i =0)  = 1-\delta. \]
    For a subset $ E $ of $ \bZ $ and $ (\xi_i)_{i \in \bZ} $, we define the following.
    \[  J \coloneqq\{ i\in E : \xi_i =1 \}  \]
    \[ K_{\Phi_0}(E) \coloneqq \sup_{\norm{a_n}_{\ell^2} \leq 1} \norm{\sum_{n \in E} a_n e^{2\pi i n x}}_{\Phi_0}, \qquad K_\Phi(J) \coloneqq \sup_{\norm{a_n}_{\ell^2} \leq 1} \norm{\sum_{n \in J} a_n e^{2\pi i n x}}_{\Phi}. \]
	If 
	\begin{equation}\label{M2cond1}
		K_{\Phi_0} (E) \lesssim \frac{|E|^{1/2}}{\Phi_0^{-1} ( \diam{E})},
	\end{equation}
	then 
	\begin{equation*}
		\bE(K_{\Phi}(J)) \lesssim 1.
	\end{equation*}
\end{thm}
Theorem \ref{M1} and \ref{M2} imply the following.
\begin{thm}\label{M3}
	Assume that $ \Phi_0$ and $\Phi_1 $ are nice Young functions such that $ \Phi_0 \in \Delta_2 $, $\Phi_1 \in \Delta_2 \cap \nabla_2 $, $ 2 \leq \alpha_{\Phi_0}^\infty $ and $ \beta_{\Phi_0}^\infty < \alpha_{\Phi_1}^\infty  $. Let $ E $ be an infinite subset of $ \bZ $ and let $ E_N = E \cap ([-2N,-N] \cup [N,2N]) $ where $ N \in \bN $. If
	\begin{equation}\label{M3cond1}
		 K_{\Phi_0}(E_N) \lesssim \frac{|E_N|^{1/2}}{\Phi_0^{-1} (N)},
	\end{equation}
	then there exists a $ \Lambda(\Phi_1) $-set in $ E $ which is not a $ \Lambda(\Phi_2) $-set for any nice Young function $ \Phi_2 $ such that $ \Phi_2 \in \Delta_2 $ and
	\begin{equation}\label{M3cond2}
		 \sup_{u \geq 1} \frac{\Phi_2(u)}{\Phi_1(u)}= \infty.
	\end{equation}
\end{thm}
\begin{remark}
	(1) The proof of Theorem 1.6 essentially follows from Talagrand's argument in \cite[Section 16.6]{Tal14} and \cite[Section 19.3]{Tal21}. But, we found a simpler way to handle large values of $|\sum a_n e^{2\pi i n x}|$ without using Gin\'{e}-Zinn Theorem, see \cite[Section 19.3]{Tal21}.\\
	(2) Let $ S $ be a $ \Lambda(\Phi_1)$-set in $ E $ constructed by Theorem \ref{M3}. The set $ S $ has the maximal density in the sense that
	\[ |S \cap [-N,N]| \approx (\Phi_1^{-1}(N))^2 \]
	(3) Theorem \ref{M2} can be applied, not only to $(e^{2\pi inx})_{n \in \bZ}$, but also to uniformly bounded orthogonal systems.
\end{remark}
The following examples show that Theorem \ref{M1}, \ref{M2} and \ref{M3} are generalizations of previous results in \cite{Bour89},\cite{Bour89_2},\cite{Rud60},\cite{Tal95}.
\begin{ex}
	Let $ \Phi(u) = u^p $ for some $ p >2 $, $ \Phi_0(u) = u^2 $ and let $ E=\bZ $. By Plancherel theorem, \eqref{M3cond1} holds. Therefore, Theorem \ref{M3} implies that there exists a $ \Lambda(p) $-set which is not in a $ \Lambda(q) $-set for any $ q>p $. This was proved in \cite{Bour89}.
\end{ex}
\begin{ex}
	Let $ \Phi_0(u) = u^{p_0} $ for some $ p_0>4 $ and $ \Phi(u) = u^p $ where $ p >p_0$. Let us consider $ Q = \{ k^2 : k \in \bN \} $ and $ Q_N = Q \cap ([-2N ,N] \cup [N,2N]) $. We have
	\[ \frac{|Q_N|^{1/2}}{\Phi_0^{-1}(N)} \approx N^{1/4-1/p_0}. \]
	In \cite[(4.1)]{Bour89_2}, Bourgain proved that $ K_{\Phi_0}(Q \cap [-N,N]) \lesssim N^{1/4 - 1/p_0} $ when $ p_0 >4 $ . Therefore, \eqref{M3cond1} is satisfied and Theorem \ref{M3} implies that there exists a $ \Lambda(p) $-set in $ Q $ when $ p >4 $ and it is not a $ \Lambda(q) $-set for any $ q >p $. This was proved in \cite{Bour89_2}. However, Theorem \ref{M2} extends the Bourgain's result because Corollary 5 in \cite{Bour89_2} shows only the existences of $\Lambda(p)$-sets in $Q$ while Theorem \ref{M2} shows their existences with the probabilistic estimate.
\end{ex}
\begin{ex}\label{1ex3}
	Let us consider a function $ \Phi(x) $ such that
	\begin{equation}\label{1exPhi}
		\Phi(u)  = \left\lbrace \begin{matrix}
			cu^2 & \mathrm{if} \ u\leq u_0\\
			u^p \log^{\alpha p} (u) &\mathrm{if} \ u > u_0
		\end{matrix} \right.
	\end{equation}
	where $ p >2 $ and $ \alpha \in \bR $. We can choose $ c $ and $ u_0 $ such that $ \Phi $ is a nice Young function. By the inequality \eqref{2pabq}, Lemma \ref{propab} (1) and Lemma \ref{proppq} (2), it is easy to check that $ \Phi  \in \Delta_2 \cap \nabla_2$ and $ \alpha_{\Phi}^\infty = \beta_{\Phi}^\infty = p $ regardless of $c$ and $u_0$. Now, we can consider the Orlicz space $ L^{\Phi}(\bT) $, which is called a \textit{$ L^p(\log L)^{\alpha}(\bT) $ space} or a \textit{Zygmund space}. The reader can find more information about this space in \cite[Chapter 4, Section 6]{BS88}. The norm $\norm{\cdot}_{\Phi}$ is equivalent to the norm of $L^p(\log L)^\alpha (\bT)$ in \cite{BS88} and the proof easily follows from \cite[Proposition 1]{AR02}\\
	Note that
	\[ \Phi^{-1} ( u) \approx u^{1/p} \log^{-\alpha} (u) \]
	for sufficiently large $ u $. If $ S $ is a $ \Lambda(\Phi) $-set, then Theorem \ref{M1} implies that
	\[ |S \cap [-N, N]| \lesssim N^{2/p} \log^{-2\alpha} (N).  \]
	Let us consider $ \Phi_1(u) $ and $ \Phi_2(u) $ such that 
	\begin{equation*}
		\Phi_1(u)  = \left\lbrace \begin{matrix}
			c_1u^2 & \mathrm{if} \ u\leq u_1\\
			u^{p_1} \log^{\alpha_1 p_1} (u) &\mathrm{if} \ u > u_1
		\end{matrix} \right. , \qquad \Phi_2(u)  = \left\lbrace \begin{matrix}
		c_2u^2 & \mathrm{if} \ u\leq u_2\\
		u^{p_2} \log^{\alpha_2 p_2} (u) &\mathrm{if} \ u > u_2
	\end{matrix} \right.
	\end{equation*}
	where $ c_1,c_2,u_1 $ and $ u_2 $ were chosen as in \eqref{1exPhi} so that $ \Phi_1 $ and $ \Phi_2 $ are both  nice Young functions. If either 
	\begin{equation}\label{1ex3cond1}
		2<p_1 <p_2 \qquad \mathrm{or} \qquad 2<p_1=p_2,  \alpha_1<\alpha_2,
	\end{equation}
	then we have
	\[ \lim_{x \geq 1} \frac{\Phi_2(u)}{\Phi_1(u)} = \infty. \]
	As in previous examples, Theorem \ref{M3} implies that there exists a $ \Lambda(\Phi_1) $-set which is not a $ \Lambda(\Phi_2) $-set for any $ \Phi_2 $ that satisfies \eqref{1ex3cond1}. Also, for $ p_1, p_2 >4 $, there exists a $ \Lambda(\Phi_1) $-set in $ Q $ which is not a $ \Lambda(\Phi_2) $-set for any $ \Phi_2 $ that satisfies \eqref{1ex3cond1}.
\end{ex}
We introduce basic notations and properties of Orlicz spaces in Section \hyperref[sec:2]{2}, which will be used throughout the paper. In Section \hyperref[sec:3]{3}, we study properties of $ \Lambda(\Phi) $-sets and prove Theorem \ref{M1}. In Section \hyperref[sec:4]{4}, we prove Theorem \ref{M2} by using Talagrand's argument in \cite{Tal14} and \cite{Tal21}. Finally, in Section \hyperref[sec:5]{5}, we derive the Littlewood-Paley estimate in Orlicz spaces. Then, we combine this estimate with Theorem \ref{M1} and Theorem \ref{M2} to prove Theorem \ref{M3}.
\begin{acknowledgement}
	The author would like to thank his advisor Alex Iosevich for productive discussions of this work and encouragement. The author also would like to thank Michel Talagrand, Allan Greenleaf and Kabe Moen for helpful conversations. The author is grateful to the anonymous referee for valuable comments and suggestions.
\end{acknowledgement}  
\section{Preliminaries}\label{sec:2}
We recall some properties of Orlicz spaces. With each Young function $ \Phi $, we can associate another Young function $ \Psi : [0, \infty) \rightarrow [0, \infty] $ defined by
\[ \Psi(v) = \sup_{u\geq 0}( u|v| - \Phi(u) ). \]
The function $ \Psi $ is called a \textit{complementary function} to $ \Phi $ and $ (\Phi, \Psi) $ is called a \textit{complementary Young pair}. If $(\Phi, \Psi)$ is a complementary Young pair, for $f \in L^\Phi(\bT) $, we have
	\begin{equation}\label{2duality}
		\norm{f}_{\Phi} \approx \sup_{g} \int |fg|dx
	\end{equation}
where the supremum was taken over all $ g $ such that $ \int \Psi(|g|) dx \leq 1 $. 
\begin{lem}\cite[Section 2.1, Proposition 1]{RR91}
	Let $ (\Phi,\Psi) $ be a complementary Young pair of nice Young functions. Their inverses $ \Phi^{-1}$ and $ \Psi^{-1} $ are well
	 defined and
	\begin{equation}\label{2eqprod}
		u \leq \Phi^{-1}(u) \Psi^{-1}(u)\leq 2u.
	\end{equation} 
\end{lem}
\begin{lem}\cite[Section 3.3, Proposition 1]{RR91}
	If $ f \in L^{\Phi}(\bT) $ and $ g \in L^{\Psi}(\bT) $ where $ (\Phi, \Psi) $ is a complementary pair of Young functions, we have
	\begin{equation}\label{2Holder}
		\int|fg| dx \lesssim \norm{f}_{\Phi} \norm{g}_{\Psi}.
	\end{equation}
\end{lem}
We define the \textit{equivalence of nice Young functions} by $ \Phi_1 \sim \Phi_2 $, if there exist constants $ 0 < c_1< c_2 < \infty $ and $ u_0 \geq 0 $ such that
\begin{equation}\label{2defequiv}
	\Phi_1(c_1u) \leq \Phi_2(u) \leq \Phi_1(c_2u) \qquad \mathrm{if} \ u \geq u_0.
\end{equation}
Nice Young functions $ \Phi_1 $ and $ \Phi_2 $ are equivalent in the sense that $ \norm{f}_{\Phi_1} \approx \norm{f}_{\Phi_2} $.
\begin{lem}\label{presmall}
	Let $ \Phi_1 $ and $ \Phi_2 $ be nice Young functions.\\
	(1) Suppose that there exists $u_0>0$ such that $ \Phi_1(u) \leq \Phi_2(u) $ for $ u \geq u_0 $. Then we have $ \norm{f}_{\Phi_1} \lesssim \norm{f}_{\Phi_2} $.\\
	(2) Suppose that there exists $u_0>0$ such that $\Phi_1(u)=\Phi_2(u)$ for $u \geq u_0$. Then we have $ \norm{f}_{\Phi_1} \approx \norm{f}_{\Phi_2} $.
\end{lem}
\begin{proof}
    The statement (1) is Theorem 3 in \cite[Section 5.1]{RR91} and (2) easily follows from (1).
\end{proof}
Indices $\alpha_{\Phi}^i$ and $\beta_{\Phi}^i$ are well defined and finite not only when $\Phi$ is a nice Young function but also when 
\begin{equation}\label{2phicond}
    \Phi:[0,\infty) \rightarrow [0,\infty) \ \mathrm{is \ increasing, \ continuous, \ unbounded \ and \ } \Phi(0)=0.
\end{equation}
Thus, we have the following properties of $\alpha_{\Phi}^i$ and $\beta_{\Phi}^i$.
\begin{lem}\label{propab} Let $\Phi: [0,\infty) \rightarrow [0,\infty) $ satisfies \eqref{2phicond}.\\
	(1) If $ \Phi_1 \sim \Phi_2 $, then $ \alpha_{\Phi_1}^\infty = \alpha_{\Phi_2}^\infty $ and $ \beta_{\Phi_1}^\infty = \beta_{\Phi_2}^\infty $. If $ x_0 =0 $ in \eqref{2defequiv} and $ \Phi_1 \sim \Phi_2 $, then $ \alpha_{\Phi_1}^a = \alpha_{\Phi_2}^a $ and $ \beta_{\Phi_1}^a = \beta_{\Phi_2}^a $.\\
	(2) Let $ \Phi^{-1} $ be the inverse of $ \Phi $. Then,
	\begin{equation}\label{2invab}
		\alpha_{\Phi^{-1}}^i = \frac{1}{\beta_{\Phi}^i} \qquad  \beta_{\Phi^{-1}}^i = \frac{1}{\alpha_{\Phi}^i} \qquad \mathrm{for} \ i = a, \infty.
	\end{equation}
	(3) Let $0 < \alpha_{\Phi}^\infty < \infty$. If $u \geq 1$, for any $\epsilon>0$, we have
	\begin{equation}\label{2aint}
	    \int_1^u s^{-(\alpha_{\Phi}^\infty -\epsilon)-1} \Phi(s) ds \lesssim_\epsilon u^{-(\alpha_{\Phi}^\infty -\epsilon)} \Phi(u).
	\end{equation}
	(4) We have
	\begin{equation}\label{2eqMpf1}
	    \alpha_{\Phi}^\infty = \sup\{ p>0 : \sup_{u, t \geq 1} \frac{\Phi(u)t^p}{\Phi(tu)} < \infty \}, \qquad   \beta_{\Phi}^\infty = \inf\{ p>0 : \inf_{u, t \geq 1} \frac{\Phi(u)t^p}{\Phi(tu)} > 0 \}.
	\end{equation}
\end{lem}
\begin{proof}
    The proofs are in \cite{Mal89}. They are Theorem 11.4, Theorem 11.5, Theorem 11.8 and Theorem 11.13 respectively.
\end{proof}
We can also obtain the following relations between $M_i(t,\Phi)$ and $\alpha_{\Phi}^i$, $\beta_{\Phi}^i$.
\begin{lem}\label{2propM}
    Let $\Phi: [0,\infty) \rightarrow [0,\infty) $ satisfies \eqref{2phicond}. For any $\epsilon >0$, we have
    \begin{equation}\label{2eqM2_1}
    	\sup_{0<t \leq 1\leq u} \frac{\Phi(u)}{\Phi(u/t)t^{\alpha_\Phi^\infty - \epsilon}} \lesssim_\epsilon 1.
    \end{equation}
    Also, when $u \geq 1$,
    \begin{equation}\label{2eqM3}
        u^{\alpha_{\Phi}^\infty -\epsilon} \lesssim_\epsilon \Phi(u) \lesssim_\epsilon u^{\beta_\Phi^\infty +\epsilon}.    
    \end{equation}
\end{lem}
\begin{proof}
    When we replace $t$ by $1/t$ in \eqref{2eqMpf1}, we obtain \eqref{2eqM2_1}.\\
    If $t \geq 1$, we obtain from \eqref{2eqMpf1} that 
    \begin{equation*}
        \frac{\Phi(1) t^{\alpha_\Phi^\infty - \epsilon}}{\Phi(t) }\leq \sup_{u,t \geq 1} \frac{\Phi(u) t^{\alpha_\Phi^\infty - \epsilon}}{\Phi(tu)} \lesssim_\epsilon 1,    
    \qquad    \frac{\Phi(1) t^{\beta_\Phi^\infty + \epsilon}}{\Phi(t) }\geq \inf_{u,t \geq 1} \frac{\Phi(u) t^{\beta_\Phi^\infty - \epsilon}}{\Phi(tu)} \gtrsim_\epsilon 1.
    \end{equation*}
    Hence, we have \eqref{2eqM3}.
\end{proof}
Let us define the following indices.
	\[ p_{\Phi}^\infty \coloneqq \liminf_{u \rightarrow \infty} \frac{u \Phi'(u)}{\Phi(u)},  \qquad q_{\Phi}^\infty \coloneqq \limsup_{u \rightarrow \infty} \frac{u \Phi'(u)}{\Phi(u)}, \]
	\[ p_{\Phi}^a \coloneqq\inf_{u > 0} \frac{u \Phi'(u)}{\Phi(u)},  \qquad q_{\Phi}^a \coloneqq \sup_{u >0} \frac{u \Phi'(u)}{\Phi(u)}. \]
Unlike Lemma \ref{propab} (1), $p_{\Phi}^i$ and $q_{\Phi}^i$ does not guarantee that $p_{\Phi_1}^i = p_{\Phi_2}^i$ or $q_{\Phi_1}^i = q_{\Phi_2}^i$ when $\Phi_1 \sim \Phi_2$, see \cite[93p]{Mal89}. However, these indices have the following useful properties.
\begin{lem}\label{proppq}
	Let $\Phi$ satisfies \eqref{2phicond}.\\
	(1) For $i=a,\infty$, we have
	\begin{equation}\label{2pabq}
		p_{\Phi}^i \leq \alpha_{\Phi}^i \leq \beta_{\Phi}^i \leq q_{\Phi}^i. 	
	\end{equation}
	(2) Let $ \Phi $ be a nice Young function. The function $ \Phi $ is in $ \Delta_2 $ if and only if $  q_\Phi^\infty <\infty $ and the function $\Phi$ is in $\nabla_2$ if and only if $1 < p_{\Phi}^\infty$\\
	(3) If $ \Phi \in \Delta_2 \cap \nabla_2 $, then there exists a nice Young function $ \Phi_1 $ such that $ 1 < p_{\Phi_1}^a \leq q_{\Phi_1}^a< \infty $ and $ \Phi_1(u) = \Phi(u) $ when $  u \geq u_0 $ for some $ u_0 $.
\end{lem}
\begin{proof}
    The statement (1) is Theorem 11.11 in \cite{Mal89} and the statement (2) is Corollary 4 in Section 2.3 of  \cite{RR91}. \\
    For (3), it follows from the statement (2) that we can choose $ \epsilon >0 $ and $ u_0 $ such that $ \Phi'(u) $ is well defined at $ u_0 $ and
	\[ 1 < p_{\Phi}^\infty -\epsilon \leq \frac{u\Phi'(u)}{\Phi(u)} \leq q_{\Phi}^\infty +\epsilon < \infty \qquad \mathrm{if } \ u \geq u_0. \]
	Let $ p = \frac{u_0 \Phi'(u_0)}{\Phi (u_0)} $ and choose $ c $ such that $ \Phi(u_0) = cu_0^p $. Let
	\[ \Phi_1(u) = \left\lbrace \begin{matrix}
		cu^p & \mathrm{if} \ u \leq u_0\\
		\Phi(u) & \mathrm{if} \ u > u_0
	\end{matrix} \right. .  \]
	Since $ \Phi_1'(u_0) = \Phi'(u_0) $, $ \Phi_1 $ is a nice Young function and
	\[ 1 < p_\Phi^\infty -\epsilon \leq \frac{u\Phi'(u)}{\Phi(u)} \leq q_{\Phi}^\infty + \epsilon < \infty \]
	for all $ u>0 $.
\end{proof}
We introduce additional properties of a nice Young function $\Phi$ such that $\Phi \in \Delta_2$ or $\Phi \in \nabla_2$ that we need later.
\begin{lem}\label{2delta_2(1+delta)}\cite[Section 2.2, Theorem 3]{RR91}
	Let $ \Phi $ be a nice Young function. Then, $ \Phi \in \Delta_2 $ if and only if there eixsts $ \delta>0 $ and $ u_0 \geq 0 $ such that $ \Phi((1+\delta) u) \leq 2 \Phi(u) $ where $ u \geq u_0 $. 
\end{lem}
\begin{lem}\label{5newconv}
	 Let $ \Phi $ be a nice Young function such that $ \Phi \in \Delta_2$ and $ 1<p < \alpha_\Phi^\infty < \infty $. There exists a nice Young function $ \tilde{\Phi} $ such that $ \Phi \sim \tilde{\Phi} $, $ \tilde{\Phi} \in \Delta_2 $ and $ \tilde{\Phi}(u^{1/p}) $ is convex.
\end{lem}
\begin{proof}
    We modify the proof of Theorem 1.7 in \cite{Mal00} which considers $\alpha_{\Phi}^a$, instead of $\alpha_{\Phi}^\infty$. Let
    \begin{equation*}
        \Phi_p(u) = \left\lbrace \begin{matrix}
		u^p & \mathrm{if} \ u \leq 1,\\
		u^p + u^p \int_1^u s^{-p-1} \Phi(s)ds & \mathrm{if} \ u > 1
	\end{matrix} \right. 
    \end{equation*}
    and 
    \begin{equation*}
        \tilde{\Phi}(u) = \int_0^u \frac{\Phi_p(s)}{s}ds.
    \end{equation*}
    Then, we obtain
    \begin{equation*}
        (\tilde{\Phi}(u^{1/p}))' = \left\lbrace \begin{matrix}
		\frac{1}{p} & \mathrm{if} \ u \leq 1,\\
		\frac{1}{p} +\frac{1}{p} \int_1^{u^{1/p}} s^{-p-1} \Phi(s)ds & \mathrm{if} \ u > 1
	\end{matrix} \right. .
    \end{equation*}
    Thus, $\tilde{\Phi}(u^{1/p})$ is convex. First, we prove that $\Phi_p \approx \Phi$ when $u$ is sufficiently large.  By \eqref{2aint} and \eqref{2eqM3}, we obtain
    \begin{equation*}
            \Phi_p (u)\lesssim u^p + \Phi(u) \lesssim \Phi(u).
    \end{equation*}
    Conversely, when $u $ is sufficiently large, we obtain
    \begin{equation*}
        \Phi_p(u) \geq u^p \int_{u/2}^u s^{-p-1}\Phi(s) ds \gtrsim \Phi(u/2) \gtrsim \Phi(u).
    \end{equation*}
    In the last inequality, we used the definition of the condition that $\Phi \in \Delta_2$. When we combine the fact that $\Phi_p \approx \Phi $ with the assumption that $\Phi$ is a nice Young function such that $\Phi \in \Delta_2 $, we obtain that $\Phi_p$ is also a nice Young function such that $\Phi_p \in \Delta_2 $ and $\Phi_p \sim \Phi$.\\
    Similarly, it suffices to show that $\tilde{\Phi} \approx \Phi_p$ when $u$ is sufficiently large. Since $\Phi$ is a Young function, we have
    \begin{equation*}
        \tilde{\Phi}(u) \lesssim \int_0^u \Phi_p'(s)ds = \Phi_p(u).
    \end{equation*}
    Conversely, since $\Phi_p \in \Delta_2$, we have
    \begin{equation*}
        \tilde{\Phi}(u) \geq \int_{u/2}^u \frac{\Phi_p(s)}{s} ds \geq \Phi_p(u/2) \gtrsim \Phi_p(u).
    \end{equation*}
\end{proof}
\begin{lem}\label{5inv_est}
	Let $ \Phi_1, \Phi_2 $ be nice Young functions such that $ \Phi_1,\Phi_2 \in \Delta_2 $. If $ \sup_{u \geq 1} $ ${\Phi_2(u)} / {\Phi_1(u)} = \infty $, then $ \sup_{u \geq 1} {\Phi_1^{-1}(u)}/{\Phi_2^{-1}(u)} = \infty $.
\end{lem}
\begin{proof}
	For any large $B$, we can find sufficiently large $ u $ such that $ {\Phi_2(u)}/{\Phi_1(u) } >B $ by the assumption. Therefore, $ \Phi_1^{-1}({u}/{B}) \geq \Phi_2^{-1}(u) $. Since $ \Phi_1 \in \Delta_2 $, Lemma \ref{2delta_2(1+delta)} implies that $ \Phi_1((1+\delta) u ) \leq 2 \Phi_1(u) $ when $ u \geq u_0 $ for some $ \delta >0 $ and $ u_0 \geq 0 $. Thus, $ \Phi_1^{-1}(u) \leq  \Phi_1^{-1}(2u)/(1+\delta) $. If $ 2^k \leq  B  < 2^{k+1} $, we obtain
	\[ \Phi_2^{-1}(u) \leq \Phi_1^{-1}\bigg(\frac{u}{B}\bigg) \leq \Phi_1^{-1} \bigg( \frac{u}{2^k}\bigg) \leq \Phi_1^{-1} (u) \bigg( \frac{1}{1+\delta} \bigg)^k. \]
	Therefore,
	\[ (1+\delta)^k \leq \frac{\Phi_1^{-1}(u)}{\Phi_2^{-1}(u)}. \]
	As $ B $ increases, so does $ (1+\delta)^k $.
\end{proof}
\section{\texorpdfstring{Properties of $ \Lambda(\Phi)  $}{Lambda(phi)}-sets }\label{sec:3}
The following results are analogues of Theorem 1.4 and Theorem 3.5 in \cite{Rud60} respectively. Arguments are similar to the proofs in \cite{Rud60} except that we consider in terms of Orlicz spaces.
\begin{thm}\label{thmRud}
	Assume that $ \Phi,\Phi_1$ and $\Phi_2 $ are nice Young functions. If 
	\begin{equation}\label{3eqas1}
		\norm{f}_{\Phi} \lesssim \norm{f}_{\Phi_1}
	\end{equation}
	for some $ \Phi_1 $ such that $ \beta_{\Phi_1} < \alpha_{\Phi} $, then
	\[ \norm{f}_{\Phi} \lesssim \norm{f}_{\Phi_2} \]
	for any $ \Phi_2 $. 
\end{thm}
\begin{proof}
We can choose $ p $ and $ \epsilon> 0$ such that $ \beta_{\Phi_1} < p < p+\epsilon < \alpha_{\Phi} $. By \eqref{2eqM3}, for sufficiently large $u$, we have
\begin{equation*}
	\Phi_1(u) \lesssim u^p < u^{p+\epsilon} \lesssim \Phi(u) 
\end{equation*}
Let $ \theta= (p-1) / (p+\epsilon -1) $ and $ \theta' = \epsilon / (p+\epsilon-1) $. The constant $ \theta $ is nonzero and finite since $ 1< p<\infty $.  By Lemma \ref{presmall} (1), we have
\begin{equation}\label{3eq1}\begin{split}
		\norm{f}_{\Phi_1}^p&\lesssim \norm{f}_p^p= \int |f|^{(p+\epsilon) \theta} |f|^{\theta'} dx \\
		&\leq \bigg( \int|f|^{p+\epsilon} dx \bigg)^{\theta} \bigg(  \int |f| dx\bigg)^{\theta'} = \norm{f}_{p+\epsilon}^{(p+\epsilon)\theta} \norm{f}_1^{\theta'}.
	\end{split}
\end{equation}
Lemma \ref{presmall} (1) and \eqref{3eqas1} implies that
\begin{equation}\label{3eq2}	
	\norm{f}_{p+\epsilon} \lesssim \norm{f}_{\Phi} \lesssim \norm{f}_{\Phi_1}.
\end{equation}
By \eqref{3eq1} and \eqref{3eq2}, we obtain
\[ \norm{f}_{\Phi_1} \lesssim \norm{f}_1. \]
Since $ \Phi_2 $ is a nice Young function, $ u \lesssim \Phi_2(u) $ if $ u \geq u_0 $ for some $u_0$. Therefore, 
\[ \norm{f}_{\Phi_1} \lesssim \norm{f}_1 \lesssim \norm{f}_{\Phi_2}. \] 
\end{proof}
Theorem \ref{thmRud} implies that \eqref{1defnLambda} and \eqref{1defnLambda2} are equivalent when $ \alpha_{\Phi}^\infty >2 $.
\begin{prop}\label{thmcard}
	Assume that $ \Phi $ is a nice Young function such that $ \alpha_{\Phi}^\infty >2 $. For any positive integer $ N $, we let $ a,b \in \bZ $ such that $ b \neq 0 $ and $ A_S(N,a,b) $ be a number of terms in a set $S \cap \{ a+b, a+2b, \cdots, a+Nb \}   $. Let $ A_S(N) := \sup_{a,b \in \bZ} A_S(N,a,b) $. If $ S $ is a $ \Lambda(\Phi) $-set, then
	\[ A_S(N) \lesssim (\Phi^{-1}(N))^2. \]
\end{prop}
\begin{proof}
	Consider the Fej\'{e}r kernel $ K_N $, i.e.,
	\begin{equation}\label{Fejer}
	   K_N (x) = \sum_{|n| \leq N} \bigg( 1-\frac{|n|}{N} \bigg) e^{2\pi i nx} 
	\end{equation}
	and a set $ A = \{ a+b, a+2b, \cdots, a+Nb \} $. As in \cite{Rud60}, we claim that $A_S(N) \lesssim \norm{K_{N}}_{\Psi}^2$. Let $  m = a + \frac{N}{2}b $ or $ m = a+ \frac{N+1}{2}b $ depending on whether $ N $ is even or odd and let
	\[ Q (x) =e^{2\pi imx} K_N(bx) = \sum_{|n| \leq N} \bigg( 1-\frac{|n|}{N} \bigg) e^{2\pi i(m+nb) x }. \]
	It follows that
	\[ \widehat{Q}(n) \geq \frac{1}{2} \qquad \mathrm{for} \ n \in A. \]
	Let $ A \cap S $ consists of numbers $ n_1, \cdots, n_\alpha $ and consider a trigonometric polynomial 
	\[ f(x) = \sum_{j=1}^\alpha e^{2\pi i n_j x}. \]
	By \eqref{2Holder}, we have 
	\[ \begin{split}
		\frac{\alpha}{2}&\leq \sum_{j=1}^\alpha  \widehat{Q}(n_j) =  \int f(-x) Q(x)dx \lesssim \norm{f}_{\Phi} \norm{Q}_{\Psi}\\
		&\leq \norm{f}_{\Phi}\norm{K_N}_{\Psi}\lesssim \norm{f}_2 \norm{K_N}_{\Psi} \lesssim \sqrt{\alpha} \norm{K_N}_{\Psi}
	\end{split} \]
	where $ \Psi $ is a complementary Young function of $ \Phi $. In the second to last inequality, we used the assumption that $ S $ is a $ \Lambda(\Phi) $-set and $\alpha_{\Phi}^\infty >2$. Once we prove that
	\begin{equation}\label{3Psiest}
		\norm{K_N}_{\Psi} \lesssim \Phi^{-1}(N),
	\end{equation}
 	it implies that
 	\[ \alpha \lesssim \norm{K_N}_\Psi^2 \lesssim (\Phi^{-1}(N))^2 \]
 	for arbitrary $ A $. Therefore, $ A_S(N) \lesssim (\Phi^{-1}(N))^2 $.\\
 	Note that 
 	\[ K_N(x) \leq \left\lbrace  \begin{array}{cc}
 		N, & \mathrm{if} \ |x| \leq \frac{1}{N},\\
 		\frac{1}{Nx^2}, & \mathrm{if} \ |x| > \frac{1}{N}.
 	\end{array} \right.   \]
 	By \eqref{2eqprod}, we have
 	\begin{equation}\label{3comp}
 		\frac{1}{\Phi^{-1}(N)} \leq \frac{\Psi^{-1}(N)}{N}.
 	\end{equation}
 	Since $ \Phi $ is a nice Young function, so does $ \Psi $. If $ K \geq 1 $, by using concavity of $\Psi^{-1}$, we have
 	\begin{equation}\label{3conc}
 		 \frac{\Psi^{-1}(N)}{K}  \leq \Psi^{-1} \bigg(\frac{N}{K}\bigg).
 	\end{equation}
	For sufficiently large $ N $, \eqref{3comp} and \eqref{3conc} imply that
 	\[ \begin{split}
 		\int \Psi\bigg( \frac{K_N}{2 \Phi^{-1}(N)} \bigg) dx &\leq \int_{|x| \leq 1/N } \Psi\bigg( \frac{\Psi^{-1}(N)N}{2N} \bigg) dx +\int_{|x| \geq 1/N} \Psi\bigg( \frac{\Psi^{-1}(N)}{2 N^2x^2} \bigg) dx\\
 		&\leq \int_{|x| \leq 1/N} \Psi \bigg( \Psi^{-1} \bigg( \frac{N}{2} \bigg) \bigg)dx  +\int_{|x|\geq 1/N} \Psi\bigg( \Psi^{-1} \bigg(  \frac{N}{2N^2x^2} \bigg) \bigg) dx \\
 		&\leq \frac{1}{2} + \int_{|x|\geq 1/N} \frac{1}{2Nx^2}dx \leq 1.
 	\end{split} \]
 	Therefore, $ \norm{K_N}_{\Psi} \leq 2\Phi^{-1}(N) $.
\end{proof}
 Theorem \ref{M1} follows from Proposition \ref{thmcard}.
\section{Probabilistic estimate about \texorpdfstring{$ \Lambda(\Phi) $}{Lambda(Phi)}-sets}\label{sec:4}
In this section, we adapt Talagrand's argument in \cite[Section 16.6]{Tal14} and \cite[Section 19.3]{Tal21} to prove Theorem \ref{M2}. Let $ X $ be a Banach space of real valued functions on $ \bT $ and $ X^\ast $ be the dual space of $ X $. We denote by $ \norm{\cdot}_C $ a norm on $ X $ such that the unit ball of its dual norm is 
\[ X_{1,C}^\ast \coloneqq \{ g \in X^\ast : \norm{g}_{X^\ast} \leq 1 , \sum_{n \in E} |\widehat{g}(n)|^2 \leq C \}, \]
i.e.,
\[ \norm{f}_C = \sup_{g \in X_{1,C}^\ast }\int f (x) \overline{g(x )}dx. \]
Let $ J $ be a subset of $ E $ and denote
\[ \norm{U_J}_C \coloneqq \sup_{g \in X_{1,C}^\ast} \bigg( \sum_{n \in J} |\widehat{g}(n)|^2 \bigg)^{1/2}. \]
Consider a real number $ p \geq 2 $. A norm $ \norm{\cdot} $ of a Banach space is called \textit{$ p $-convex} if there exists a constant $ \eta >0 $ such that 
\begin{equation}\label{4defnconv}
		\norm{\frac{x+y}{2}} \leq 1- \eta \norm{x-y}^p
\end{equation}
for any $ \norm{x},\norm{y}\leq 1 $. Note that $\norm{\cdot}_p $ is $ 2 $-convex if $ p \leq 2 $. See \cite[63p]{LinTza}, \cite{Han56}.
\begin{lem}\label{TalProb}{\cite[Theorem  16.5.3]{Tal14}}, {\cite[Theorem  19.3.8]{Tal21}}
	Consider a Banach space $ X $ of real valued functions on $ \bT $ with a norm $ \norm{\cdot} $ such that $ X^\ast $ is $ 2 $-convex with corresponding constant $ \eta $ in \eqref{4defnconv}. Suppose that $ \norm{e^{2\pi i n x}} \leq 1 $ for all $ n \in E $. Then, there exists a number $ K(\eta) $ that depends only on $ \eta $ with the following property. Consider a number $ C $ and define $ B = \max(C, K(\eta)\log |E|) $. Consider a number $ \delta>0 $ and assume that for some $ \epsilon >0 $,
	\[ \delta \leq \frac{1}{B\epsilon |E|^\epsilon } \leq 1. \]
	For random variables $ (\xi_i)_{i \in E} $ and $J=\{i \in E : \xi_i=1 \}$ in Theorem \ref{M2}, we have
	\[ \bE(\norm{U_J}_C^2) \leq \frac{K(\eta)}{\epsilon}. \]
\end{lem}
\begin{prop} \label{4propC_1C_2}
	Let $ p_1 > \beta_{\Phi}^\infty $ and $ X = L^{p_1}(\bT) $. If all the assumptions of Theorem \ref{M2} hold and
	\[  C_1 = \frac{|E|^{1/2}}{\Phi^{-1}(\diam{E})}, \qquad C_2 = \frac{|E|\log |E|}{10(\Phi^{-1}(\diam{E}))^2} , \]
	then
	\begin{equation}\label{4expU_J}
		\bE( \norm{U_J}_{C_1}) \lesssim 1, \qquad \bE (\norm{U_J}_{C_2}) \lesssim \sqrt{\log |E|}.
	\end{equation}
\end{prop}
\begin{proof}
	Since $ p_1 >2 $, $ X^\ast = (L^{p_1}(\bT))^\ast $ is $ 2 $-convex. Since $1 \lesssim  K_{\Phi_0}(E) $, \eqref{M2cond1} implies that
	\begin{equation}\label{4sizePhi_0}
		\Phi_0^{-1}(\diam{E}) \lesssim |E|^{1/2}.
	\end{equation}
	We can choose $ p_0 >2 $ and $ \epsilon_0>0 $ such that $ \beta_{\Phi_0}^\infty < p_0 < p_0+\epsilon_0  < \alpha_{\Phi}^\infty $. By \eqref{2eqM3}, we have
	\begin{equation}\label{4sizecomp}
		\Phi^{-1} (u) \lesssim u^{1/(p_0+\epsilon_0)} \leq u^{1/p_0} \lesssim \Phi_0^{-1}(u)
	\end{equation}
	for sufficiently large $ u $. From \eqref{4sizePhi_0} and \eqref{4sizecomp}, we obtain
	\[ C_1 \gtrsim  \frac{\Phi_0^{-1}(\diam{E})}{\Phi^{-1}(\diam{E})} \gtrsim (\diam{E})^{1/p_0 - 1/(p_0+\epsilon_0)} \gtrsim |E|
	^{1/p_0 - 1/(p_0+\epsilon_0)} \gtrsim \log |E| \]
	when $|E|$ is sufficiently large.
	Therefore,  $ B \approx C_1 $ in Lemma \ref{TalProb}. 
	From \eqref{4sizePhi_0} and \eqref{4sizecomp} again, we obtain 
	\begin{equation}\label{4sizecomp2}
		\Phi^{-1}(\diam{E})\lesssim \diam{E}^{1/(p_0+\epsilon_0)}\lesssim (\Phi_0^{-1}(\diam{E}))^{p_0/(p_0+\epsilon_0)} \lesssim |E|^{{p_0}/{2(p_0+\epsilon_0)}}.
	\end{equation}
	We can choose $ \epsilon_1 $ such that 
	\begin{equation}\label{4epsilon1}
		\frac{1}{2} - \frac{p_0}{2(p_0+
			\epsilon_0)} -\epsilon_1 > 0.
	\end{equation}
	Let the left hand side of \eqref{4epsilon1} be $ \epsilon $. By \eqref{4sizecomp2}, we have
	\[ \frac{1}{B\epsilon |E|^\epsilon} \gtrsim \frac{(\Phi^{-1}(\diam{E}))^2}{|E|} |E|^{\epsilon_1} \geq \delta. \]
	It follows from Lemma \ref{TalProb} that
	\[ \bE\norm{U_J}_{C_1} \lesssim 1. \]
	Similarly, by \eqref{4sizePhi_0} and \eqref{4sizecomp}, we have
	\[ \begin{split}
		C_2 &\gtrsim \frac{\log |E|}{10} \bigg( \frac{\Phi_0^{-1}(\diam{E})}{\Phi^{-1}(\diam{E})} \bigg)^2 \gtrsim \frac{\log |E|}{10} (\diam{E})^{2 (1/p_0 -1/(p_0 +\epsilon_0) )}\\
		&\gtrsim \frac{\log |E|}{10}|E|^{2 (1/p_0 - 1/(p_0+\epsilon_0))}  \gtrsim \log |E|.
	\end{split}\]
	Thus, $ B\approx C_2 $ in Lemma \ref{TalProb}. If $ \epsilon = 1/\log |E| $, then
	\[ \frac{1}{B\epsilon |E|^\epsilon} = \frac{10 (\Phi^{-1} (\diam{E}))^2}{e |E|} \geq \delta. \]
	Lemma \ref{TalProb} implies that
	\[ \bE\norm{U_J}_{C_2} \lesssim \sqrt{\log |E|}. \]
\end{proof}
\begin{lem}\label{4estfPhi}
    Let $ p_1 > \beta_{\Phi}^\infty $ and $\Phi \in \Delta_2$, $ X = L^{p_1}(\bT) $ and assume that \eqref{M2cond1} holds and $C_1$ and $C_2$ are constants in Proposition \ref{4propC_1C_2}. If $f$ is a measurable function such that $ \norm{f}_{C_1} \leq 1 $ and $ \norm{f}_{C_2} \leq \sqrt{\log |E|}  $, then $ \norm{f}_{\Phi} \lesssim 1 $.
\end{lem}
\begin{proof}
	We can choose $p$ such that $ \beta_{\Phi}^\infty < p  < p_1$. We have
	\begin{equation}\label{4normfphi1}
	\norm{f}_{\Phi} \leq \norm{f \textbf{1}_{\{|f| \leq 1 \} }}_{\Phi} + \norm{f \textbf{1}_{ \{1 \leq |f| \leq D \} }}_{\Phi}  +\norm{f \textbf{1}_{\{|f| \geq D \} }}_{\Phi}
\end{equation}
where $ D^{p_1-2} \approx C_1 $. Such $D$ exists since $C_1 \gtrsim 1$ and $p_1 >2$. It is easy to show that $ \norm{f \textbf{1}_{\{|f| \leq 1 \} }}_{\Phi}  \lesssim 1 $. When $k \geq 1$, $\Phi(u/k) \leq \Phi(u)/k$ since $\Phi$ is convex and $\Phi(0)=0$. Thus, we obtain by Fubini's theorem that
\[ \begin{split}
	\norm{f \textbf{1}_{ \{1 \leq |f| \leq D \}}}_{\Phi} &\leq \inf\{ k \geq 1: \int \Phi(|f \textbf{1}_{ \{1 \leq |f| \leq D \}}/k|) dx \leq 1 \}\\
	&\leq \inf\{ k \geq 1: \int \frac{1}{k}\Phi(|f \textbf{1}_{ \{1 \leq |f| \leq D \}}|) dx \leq 1 \}\\
	& = \inf\{ k \geq 1 : \int_1^D \frac{1}{k}\Phi'(u)\mu \{ |f| \geq u  \}du \leq 1  \}\\
	& = \max (\int_1^D \Phi'(u)\mu \{ |f| \geq u  \}du, 1 )
\end{split} \]
where $ \mu  $ is the Lebesgue measure. By Lemma \ref{proppq} (2) and \eqref{2eqM3}, when $u \geq 1$,
\begin{equation*}
    \Phi'(u) \lesssim \frac{\Phi(u)}{u} \lesssim u^{p-1}
\end{equation*}
Therefore,
\begin{equation}\label{4estf2}
		\int_1^D \Phi'(u) \mu\{ |f| \geq u  \} du  \lesssim \int_1^D u^{p-1} \mu\{ |f| \geq u  \}du.
\end{equation}
As mentioned in \cite[557p]{Tal14} and \cite[Lemma 19.3.11]{Tal21}, Hahn-Banach theorem and the fact that $ \norm{f}_{C_1} \leq 1 $ implies that there exist $ v_1  $ and $ v_2 $ such that $ f = v_1 + v_2 $ and
\begin{equation}\label{4decomp}
   \norm{v_1}_{p_1} \leq 1, \qquad v_2 = \sum_{n \in E} a_n e^{2\pi i n x} \qquad \mathrm{where} \ \sum_{n \in E} |a_n|^2 \leq C_1^{-1}. 
\end{equation}
By the fact that $ \norm{v_1}_{p_1} \leq 1 $, we have 
\begin{equation}\label{4meas_u_1}
	\mu\{ |v_1 | >u \} \leq u^{-p_1}.
\end{equation}
Similarly, $ \norm{v_2}_2 \leq C_1^{-1/2} $ implies that
\begin{equation}\label{4meas_u_2}
	\mu \{ |v_2 | >u \} \leq C_1^{-1} u^{-2}.
\end{equation}
If $ u \leq D $, then
\begin{equation}\label{4meas_f}
	\mu\{ |f| \geq u \} \leq \mu \{  |v_1| >\frac{u}{2} \}  + \mu\{ |v_2 | > \frac{u}{2} \} \lesssim u^{-p_1} + C_1^{-1} u^{-2} \lesssim u^{-p_1}.
\end{equation}
By \eqref{4estf2}, \eqref{4meas_f} and the fact that $ p < p_1 $, we obtain
\[ \norm{f\textbf{1}_{\{1 \leq |f| \leq D \}}}_{\Phi} \lesssim \max ( \int_1^D  u^{p-p_1-1} du ,1 ) \lesssim 1. \]
Similarly, $ \norm{f}_{C_2} \leq \sqrt{\log |E|} $ implies that there exist $ w_1 $ and $ w_2 $ such that $ f = w_1+w_2 $ and
\begin{equation*}
   \norm{w_1}_{p_1} \leq \sqrt{\log |E|}, \qquad w_2 = \sum_{n \in E} a_n e^{2\pi i nx} 
\end{equation*}
where
\begin{equation*}
    \sum_{n\in E} |a_n|^2 \leq \frac{\log |E|}{C_2}\approx \frac{(\Phi^{-1}(\diam{E}))^2}{|E|}.
\end{equation*}
Thus, we have
\begin{equation}\label{4measv_1}
	\mu\{ |w_1| >u \} \leq (\log |E|)^{p_1/2}u^{-p_1}.
\end{equation}
By \eqref{M2cond1} and the definition of $ w_2 $, we have
\begin{equation}\label{4estv_2}
		\norm{w_2}_{\Phi_0} \lesssim \frac{|E|^{1/2}}{\Phi_0^{-1}(\diam{E})} \norm{w_2}_2 \lesssim \frac{\Phi^{-1}({\diam{E}})}{\Phi_0^{-1}(\diam{E}) }.
\end{equation}
Let 
\[ k_0= C \frac{\Phi^{-1}({\diam{E}})}{\Phi_0^{-1}(\diam{E}) } \]
where $ C $ is a sufficiently large constant. It follows from \eqref{4estv_2} that
\begin{equation}\label{4measv_2}
	\mu\{ |w_2| >u \} \leq \frac{1}{\Phi_0(u/k_0)}.
\end{equation}
Using \eqref{4measv_1} and \eqref{4measv_2}, we have
\begin{equation}\label{4meas_f2}
	\begin{split}
		\mu\{ |f| >u \} &\leq \mu\{ |w_1 | > u/2 \} + \mu \{ |w_2| > u/2 \}\\
		&\lesssim u^{-p_1} (\log |E|)^{p_1/2}  + \frac{1}{\Phi_0(u/k_0)}.
	\end{split}	
\end{equation}
We also have
\begin{equation}\label{4maxv_2}
	\norm{w_2}_\infty \leq \bigg( \sum_{n \in E} |a_n|^2 \bigg)^{1/2}  \bigg( \sum_{n \in E} |e^{2\pi i nx}|^2 \bigg)^{1/2} \lesssim  \Phi^{-1}(\diam{E}).
\end{equation} 
From \eqref{4meas_f2} and \eqref{4maxv_2}, it follows that
\begin{equation}\label{4estf3}
	\begin{split}
		\norm{f\textbf{1}_{\{|f| \geq D \} } }_{\Phi} &=\max ( \int_D^\infty \Phi'(u) \mu\{ |f| \geq u  \} du , 1  ) \\
		&\lesssim  \max ( \int_D^\infty \Phi'(u) u^{-p_1} (\log |E|)^{p_1/2} du + \int_D^{D_1} 
		\frac{\Phi'(u)}{\Phi_0(u/k_0)} du , 1 )
	\end{split}
\end{equation}
where $ D_1 \approx \Phi^{-1}(\diam{E})$. As in \eqref{4estf2},
\begin{equation}\label{41stint}
	\begin{split}
		(\log |E|)^{p_1/2} \int_D^\infty \Phi' (u) u^{-p_1} du &\lesssim (\log |E|)^{p_1/2} \int_D^\infty u^{p-p_1-1}du\\
		&\lesssim (\log |E|)^{p_1/2} D^{p-p_1} \lesssim 1 .
	\end{split}
\end{equation}
For the second integral, we have
\begin{equation}\label{4intDD_1}
	\begin{split}
		\int_D^{D_1} 
		\frac{\Phi'(u)}{\Phi_0(u/k_0)} du &\lesssim \int_D^{D_1}  \frac{\Phi(u)}{u\Phi_0(u/k_0)}du \\
		&\lesssim \int_{D/D_1}^1 \frac{\Phi(\Phi^{-1}(\diam{E})s)}{s\Phi_0(C^{-1}\Phi_0^{-1}(\diam{E})s)}ds \qquad \mathrm{where} \ u = D_1s.
	\end{split}
\end{equation}
In the second inequality, we used the fact that $\Phi \in \Delta_2 $. We can choose $\epsilon_0 > 0$ such that $\beta_{\Phi_0}^\infty < \beta_{\Phi_0}^\infty +\epsilon_0 < \alpha_{\Phi}^\infty$. Since $s\leq 1$ and $s \Phi^{-1}(\diam{E}) \geq D $, we obtain from \eqref{2eqM2_1} that 
\begin{equation}\label{4esta_Phi}
    \Phi(s \Phi^{-1}(\diam{E}))\lesssim_{\epsilon_0} s^{\beta_{\Phi_0}^\infty + \epsilon_0} \diam{E}
\end{equation}
Now, we can choose $\epsilon_1$ such that 
\begin{equation}\label{4epsilon_1}
	\frac{1}{\beta_{\Phi_0}^\infty} - \frac{1}{\alpha_\Phi^\infty } > 2\epsilon_1 \qquad \mathrm{and} \qquad \epsilon_0 > \frac{\epsilon_1 (\beta_{\Phi_0}^\infty)^2}{1-\epsilon_1 \beta_{\Phi_0}^\infty}.
\end{equation}
By \eqref{2invab}, we obtain that 
\begin{equation*}
	s \Phi_0^{-1}(\diam{E}) \geq \bigg(\frac{D}{D_1}\bigg) \Phi_0^{-1}(\diam{E}) \gtrsim \bigg( \frac{\Phi_0^{-1}(\diam{E}) }{\Phi^{-1}(\diam{E})}\bigg)^{{(p_1-1)}/{(p_1-2)}}
\end{equation*}
and
\begin{equation*}
	 \frac{\Phi_0^{-1}(\diam{E} )}{\Phi^{-1}(\diam{E})} \gtrsim_{\epsilon_1} \diam{E}^{(\beta_{\Phi_0}^\infty)^{-1} - (\alpha_\Phi^\infty)^{-1} -2\epsilon_1}.
\end{equation*}
Since the exponent of $\diam{E}$ is positive, $\Phi^{-1}(\diam{E})s$ is sufficiently large. Therefore, $ C^{-1} \Phi_0(s\Phi_0^{-1} (\diam{E}) )\lesssim \Phi_0(C^{-1}\Phi_0^{-1}(\diam{E})s)$ by Lemma \ref{2delta_2(1+delta)}.
Now, we will show that 
\begin{equation}\label{4estb_Phi_0_1}
	 \Phi_0^{-1}(s^{\beta_{\Phi_0}^\infty/ (1-\epsilon_1 \beta_{\Phi_0}^\infty) } \diam{E}) \lesssim_{\epsilon_1} s \Phi_0^{-1}( \diam{E}).
\end{equation}
We obtain from \eqref{2invab} and \eqref{4epsilon_1} that
\begin{equation*}
	\begin{split}
		s( \diam{E})^{(\beta_{\Phi_0}^\infty)^{-1} - \epsilon_1} &\gtrsim \bigg(\frac{D}{D_1}\bigg) \diam{E}^{(\beta_{\Phi_0}^\infty)^{-1} - \epsilon_1}\\
		&\gtrsim_{\epsilon_1} \diam{E}^{\frac{p_1-1}{p_1-2}( (\beta_{\Phi_0}^\infty)^{-1} - (\alpha_\Phi^\infty)^{-1} -2\epsilon_1)} \gtrsim 1.
	\end{split}
\end{equation*}
Therefore, $s^{\beta_{\Phi_0}^\infty/ (1-\epsilon_1 \beta_{\Phi_0}^\infty) } \diam{E} \gtrsim 1$. Since $s<1$, it follows from \eqref{2eqM2_1} that 
\begin{equation*}
	\Phi_0^{-1}(s^{\beta_{\Phi_0}^\infty/ (1-\epsilon_1 \beta_{\Phi_0}^\infty) } \diam{E}) \lesssim_{\epsilon_1} s^ {\beta_{\Phi_0}^\infty/ (1-\epsilon_1 \beta_{\Phi_0}^\infty) (\alpha_{{\Phi_0}^{-1}}^\infty -\epsilon_1 )}\Phi_0^{-1}(\diam{E})= s \Phi_0^{-1}(\diam{E}).
\end{equation*}
Thus, we established \eqref{4estb_Phi_0_1}. Inequality \eqref{4intDD_1} combined with \eqref{4esta_Phi} and \eqref{4estb_Phi_0_1} gives
\begin{equation}\label{42ndint}
	\int_D^{D_1} 
	\frac{\Phi'(u/k)}{\Phi_0(u/k_0)} du \lesssim_{\epsilon_0,\epsilon_1} \int_0^1 s^{\epsilon_0-1 -\epsilon_1 (\beta_{\Phi_0}^\infty)^2 / (1-\epsilon_1 \beta_{\Phi_0}^\infty ) }ds \lesssim 1.
\end{equation}
The integral is bounded by the choice of $\epsilon_1$.
Combining \eqref{4estf3} with \eqref{41stint} and \eqref{42ndint} gives 
\[ \norm{f\textbf{1}_{\{|f| \geq D \} }}_{\Phi}  \lesssim 1.  \]
\end{proof}
Now, we can prove Theorem \ref{M2}
\begin{proof}
	Let us assume that $ \sum_{n \in J} |a_n|^2 \leq 1 $. When $X = L^{p_1}(\bT)$ such that $p_1 > \beta_{\Phi}^\infty$, we have
	\[ \begin{split}
		\norm{\sum_{n \in J} a_n e^{2\pi i nx}}_{C_1} &= \sup_{g \in X_{1,C_1}^\ast} \int \sum_{n \in J} a_n e^{2\pi i n x} g(x) dx = \sup_{g \in X_{1,C}^\ast} \sum_{n \in J} a_n \widehat{g}(-n)\\
		&\leq \sup_{g \in X_{1,C_1}^\ast} \bigg( \sum_{n \in J} |a_n|^2 \bigg)^{1/2} \bigg( \sum_{n \in J} |\widehat{g}(-n) |^2\bigg)^{1/2} \leq \norm{U_J}_{C_1}.
	\end{split} \]
Similarly, from $ \sum_{n \in J} |a_n|^2 \leq 1 $, we have
\[ \norm{\sum_{n \in J} a_n e^{2\pi i n x}}_{C_2} \leq \norm{U_J}_{C_2}. \]
Let 
\[ h(x) := \frac{\sum_{n \in J}a_n e^{2\pi i n x}}{\norm{U_J}_{C_1} +( {\norm{U_J}_{C_2}}/{\sqrt{\log |E|}} )} \]
Since $ \norm{h}_{C_1} \leq 1 $ and $ \norm{h}_{C_2} \leq \sqrt{\log |E|} $, Lemma \ref{4estfPhi} implies that $ \norm{h}_{\Phi} \lesssim 1 $. Therefore,
\[ K_{\Phi}(J)= \sup_{\norm{a_n}_{\ell^2}\leq 1}\norm{\sum_{n \in J} a_n e^{2\pi i n x}}_{\Phi} \lesssim  \norm{U_J}_{C_1} + \frac{\norm{U_J}_{C_2}}{\sqrt{\log |E|}}. \]
Applying \eqref{4expU_J} gives $ \bE(K_{\Phi}(J)) \lesssim 1. $
\end{proof}
\section{Existence of \texorpdfstring{$ \Lambda(\Phi) $}{Lambda(Phi)}-sets}\label{sec:5}
Before we prove Littlewood-Paley estimate on Orlicz space, we have to start with classical results. Although Littlewood-Paley estimate of Fourier series in $ L^p(\bT) $ is well known (see e.g., \cite[Theorem 14.2.1]{Zyg02}), we need the following proof of the estimate which uses transference of multipliers.\\
\indent Let $\xi_0 \in \bR$. A bounded function $b$ on $\bR$ is called \textit{regulated at the point $\xi_0$} if
    \begin{equation*}\label{5defreg}
        \lim_{\epsilon \rightarrow 0} \frac{1}{\epsilon} \int_{|\xi| \leq \epsilon} (b(\xi_0-\xi) - b(\xi_0))d\xi =0.
    \end{equation*}
\begin{lem}\cite[Theorem 4.3.7]{LG14}\label{5trans}
    Suppose that $b$ is a regulated function at every point in $\bZ$. Let $T_1$ be a linear operator defined on $L^p(\bR)$ for some $1<p<\infty$ such that
    \begin{equation*}
        \widehat{T_1f}(\xi) = b(\xi)\widehat{f}(\xi)
    \end{equation*}
    and $\norm{T_1}_{L^p(\bR) \rightarrow L^p(\bR)}$ is finite. Let us consider another linear operator $T_2 $ associated with $b$ defined on $L^p(\bT)$ such that
    \begin{equation*}
        T_2f(x) = \sum_{n \in \bZ} b(n) a_n e^{2\pi i n x} 
    \end{equation*}
    where $f(x) = \sum_n a_n e^{2\pi in x}$. Then, we have
    \begin{equation*}
        \norm{T_2}_{L^p(\bT) \rightarrow L^p(\bT)} \lesssim \norm{T_1}_{L^p(\bR) \rightarrow L^p(\bR)}. 
    \end{equation*}
\end{lem}
Let $\psi(x)$ be a compactly supported smooth function and let 
\begin{equation*}
    \psi_0(\xi)= \psi(\xi), \qquad \psi_j(\xi) = \psi\bigg(\frac{\xi}{2^{j+1}}\bigg) - \psi\bigg(\frac{\xi}{2^j} \bigg) \qquad \mathrm{for} \ j \in \bN.
\end{equation*}
Let $r_j(t)$ be the Rademacher function where $t \in [0,1]$. We can define linear operators
\begin{equation*}
    \widehat{P_jf}(\xi) = \psi_j(\xi )\widehat{f}(\xi)
\end{equation*}
and
\begin{equation*}
    {T_1f}(x) = \sum_{j=0}^\infty P_jf(x) r_j(t) 
\end{equation*}
where $f(x) $ is a complex-valued measurable function on $\bR$. We also define corresponding linear operators such that 
\begin{equation*}
    {\tilde{P}}_{j}f(x) = \sum_j \psi_j(n) a_n e^{2\pi in x}
\end{equation*}
and
\begin{equation*}
    T_2f(x) = \sum_{j=0}^\infty \tilde{P}_j f(x) r_j(t)
\end{equation*}
where $f (x) = \sum_n a_n e^{2\pi i n x}$.
\begin{lem}
	For any $ 1 < p < \infty $, 
	\begin{equation}\label{5classic} 
		\norm{\sum_{j=0}^\infty \tilde{P}_jf(x) r_j(t) }_{p} \leq C(p) \norm{f}_{p}  
	\end{equation}
	with $C(p)$ independent of $t$.
\end{lem}
\begin{proof}
    It is well known that $\norm{T_1f}_{L^p(\bR)} \lesssim \norm{f}_{L^p(\bR)}$, see e.g., \cite{St70}[Section 4.5, 5.3.2].
Once we show that $\sum_{j =0}^\infty \psi_j(\xi) r_j(t)$ is regulated everywhere on $\bR$, \eqref{5classic} follows from Lemma \ref{5trans}. Let $\xi_0 \in \bR$, then we obtain
\begin{equation*}
    \begin{split}
        \frac{1}{\epsilon} & \bigg|\int_{|\xi| \leq \epsilon} \sum_{j=0}^{\infty} \psi_j(\xi_0 - \xi) r_j(t) -\sum_{j=0}^{\infty} \psi_j(\xi_0)r_j(t)d\xi \bigg|\\
        &\lesssim\frac{1}{\epsilon} \int_{|\xi | \leq \epsilon } \sum_{j=0}^\infty \bigg|\psi_j(\xi_0 -\xi) -\psi_j(\xi_0) \bigg| d\xi \\
        &\lesssim \frac{1}{\epsilon} \int_{|\xi| \leq \epsilon} \sum_{j=0}^\infty \frac{|\xi|}{2^j} d\xi \leq \epsilon.
    \end{split}
\end{equation*}
Since $\xi_0$ is arbitrary, $\sum_{j=0}^\infty \psi_j(\xi) r_j(t)$ is regulated everywhere on $\bR$. 
\end{proof}
Now, let us introduce an interpolation lemma in Orlicz spaces. Let $ \cP $ be a set of functions $ \rho : [0 ,\infty) \rightarrow [0,\infty) $ such that it is continuous and positive on $ (0, \infty) $ and 
\[ \rho(s) \leq \max (1 , s/t) \rho(t) \qquad \mathrm{for \ any \ } s,t >0. \]
By $ \cP^\pm $, we denote a set of $ \rho \in \cP $ such that
\begin{equation}\label{2condpm}
	\lim\limits_{t \rightarrow 0} \overline{\rho}(t) = \lim\limits_{t \rightarrow \infty} \frac{\overline{\rho}(t)}{t}=0 \qquad \mathrm{where} \ \overline{\rho}(t) \coloneqq \sup_{s>0} \frac{\rho(st)}{\rho(s)}.
\end{equation}
Note that \eqref{2condpm} is equivalent to $ 0 < \alpha_\rho^a \leq \beta_\rho^a <1 $.
\begin{lem}\cite[Theorem 14.8]{Mal89}\label{5intplem}
	Let $ \Phi_0, \Phi_1 $ and $ \Psi_0 $, $ \Psi_1 $ be Young functions. Assume that $ T $ is a continuous linear operator such that $ \norm{T}_{L^{\Phi_0}(\bT) \rightarrow L^{\Psi_0}(\bT)}  $ and $ \norm{T}_{L^{\Phi_1}(\bT) \rightarrow L^{\Psi_1}(\bT)}  $ are both finite. If
	\[ \Phi^{-1} = \Phi_0^{-1} \rho (\Phi_1^{-1} / \Phi_0^{-1}) \qquad \mathrm{and} \qquad  \Psi^{-1} = \Psi_0^{-1} \rho(\Psi_1^{-1} / \Psi_0^{-1}) \]
	with $ \rho \in \cP^\pm $, then 
	\begin{equation}\label{2intp}
		\norm{T}_{L^{\Phi}(\bT) \rightarrow L^{\Phi}(\bT)} \lesssim \max \{ \norm{T}_{L^{\Phi_0}(\bT) \rightarrow L^{\Psi_0}(\bT)}, \norm{T}_{L^{\Phi_1}(\bT) \rightarrow L^{\Psi_1}(\bT)} \}.
	\end{equation}
\end{lem}
Lemma \ref{5intplem} was proved by Gustavsson and Peetre in \cite{GP77} and we used notations in \cite{Mal89}.
\begin{lem}
	Let $ r_j(t) $ be the Rademacher function where $ t \in \bT $. If $ \Phi $ is a nice Young function such that $ \Phi \in \Delta_2 \cap \nabla_2 $, then
	\begin{equation}\label{5LPintp}
		\norm{\sum_{j \in \bN} \tilde{P}_jf(x) r_j(t) }_{\Phi} \leq C(\Phi) \norm{f}_{\Phi} 
	\end{equation}
	where $C(\Phi)$ only depends on $\Phi$, not $t$.
\end{lem}
\begin{proof}
	The argument is similar to the proof of Corollary 2 in Chapter 14 of \cite{Mal89}. By Lemma \ref{proppq} (3), there exists $ \Phi_1 $ such that $ 1 < p_{\Phi_1}^a \leq q_{\Phi_1}^a < \infty $ and $ \Phi_1(u) = \Phi(u) $ if $ u\geq u_0 $ for some $ u_0 $. The inequality \eqref{2pabq} implies that $ 1 < \alpha_{\Phi_1}^a \leq \beta_{\Phi_1}^a <\infty $. We can choose $ p_0, p_1 $ such that $ 1 < p_0 < \alpha_{\Phi_1}^a \leq \beta_{\Phi_1}^a < p_1  < \infty $. Put
	\[ \rho(t) = t^{\frac{p_1}{p_1-p_0}} \Phi_1^{-1} (t^{-\frac{p_0p_1}{p_1-p_0}}) \qquad \mathrm{if} \ t >0   \]
	and $ \rho(0)=0 $. By \eqref{2invab}, we have
	\[ \alpha_\rho^a = \lim\limits_{t \rightarrow 0^+} \frac{\log M^\infty (t,\Phi)}{\log t} = \frac{p_1}{p_1-p_0} -\frac{p_1p_0}{p_1-p_0} \beta_{\Phi^{-1}}^a = \frac{p_1}{p_1-p_0} \bigg( 1- \frac{p_0}{\alpha_{\Phi}^a} \bigg) > 0 \]
	and
	\[ \beta_\rho^a = \lim\limits_{t \rightarrow \infty} \frac{\log M^\infty (t, \Phi)}{\log t} = \frac{p_1}{p_1-p_0} \bigg( 1- \frac{p_0}{\beta_{\Phi}^a} \bigg)  <\frac{p_1}{p_1-p_0} \bigg( 1- \frac{p_0}{p_1} \bigg) =1. \]
	Therefore, $ \rho \in \cP^{\pm} $. It follows from \eqref{5classic} that there exist constants $ C_1 $ and $ C_2 $ such that
	\[ \norm{\sum_{j \in \bN} \tilde{P}_jf(x) r_j(t) }_{p_0} \leq C_0 \norm{f}_{p_0 } \qquad  \norm{\sum_{j \in \bN} \tilde{P}_jf(x) r_j(t) }_{p_1} \leq C_1 \norm{f}_{p_1 }.  \]
	Since $ \Phi_1^{-1}(u) = u^{1/p_0} \rho(u^{1/p_1-1/p_0}) $, \eqref{2intp} and Lemma \ref{presmall} (2) implies that 
	\[ 	\norm{\sum_{j \in \bN} \tilde{P}_jf(x) r_j(t) }_{\Phi} \approx \norm{\sum_{j \in \bN} \tilde{P}_jf(x) r_j(t) }_{\Phi_1} \lesssim \max( C_0,C_1 )\norm{f}_{\Phi_1} \approx \norm{f}_{\Phi}. \]
\end{proof}
\begin{lem}\label{LPest}
	(Littlewood-Paley estimate in Orlicz spaces) Assume that $ \Phi $ is a nice Young function such that $ \Phi \in \Delta_2 \cap \nabla_2 $. Then, we have
	\begin{equation}\label{5LPapprox}
		\norm{\bigg( \sum_{j \in \bN} |\tilde{P}_j f |^2  \bigg)^{\frac{1}{2}} }_{\Phi} \approx \norm{f}_{\Phi}
	\end{equation}
\end{lem}
\begin{proof}
	We follow the proof of classical Littlewood-Paley estimate with slight modifications. First, we prove that
	\begin{equation}\label{5LPs}
		\norm{\bigg( \sum_{j \in \bN} |\tilde{P}_jf|^2 \bigg)^{1/2}}_\Phi \lesssim \norm{f}_{\Phi}.
	\end{equation}
	By Khintchine's inequality, we have
	\[ \norm{ \bigg( \sum_{j \in \bN} |\tilde{P}_j f|^2 \bigg)^{1/2} }_{\Phi} \approx \norm{\norm{ \sum_{j \in \bN} \tilde{P}(x) r_j(t) }_{L^1(dt)}}_{\Phi(dx)}. \]
	By Jensen's inequality, we obtain
	\[ \begin{split}
	\norm{ \bigg( \sum_{j \in \bN} |\tilde{P}_jf|^2 \bigg)^{1/2} }_{\Phi} &\approx \inf\bigg\{ k :  \int \Phi \bigg(\frac{1}{k}\int \bigg|\sum_{j \in \bN} \tilde{P}_jf(x) r_j(t) \bigg|dt \bigg)dx \leq 1 \bigg\}\\
	&\leq \inf\bigg\{ k :  \iint \Phi \bigg(\frac{1}{k} \bigg|\sum_{j \in \bN} \tilde{P}_jf(x) r_j(t) \bigg| \bigg)dxdt \leq 1 \bigg\}.
	\end{split} \]
	Let $ k = C\norm{f}_{\Phi} $, then \eqref{5LPintp} implies that
	\[ \int \Phi \bigg(\frac{1}{k} \bigg|\sum_{j \in \bN} \tilde{P}_jf(x) r_j(t) \bigg| \bigg)dx \leq 1  \]
	for some constant $ C $. Therefore, \eqref{5LPs} was proved. For the converse, we use \eqref{2duality} and \eqref{2Holder} and obtain that, for some finite constant $c$,
	\[ \begin{split}
		\norm{f} &\lesssim
		 \sup_{\norm{g}_{\Psi} \leq 1} \int f \overline{g} dx \lesssim \sup_{\norm{g}_{\Psi} \leq 1} \sum_{i \leq |c|} \int \sum_j \tilde{P}_jf \tilde{P}_{j+i} \overline{g} dx\\
		&\lesssim \sup_{\norm{g}_{\Psi}\leq 1 } \int \bigg( \sum_j  |\tilde{P}_jf|^2 \bigg)^{1/2} \bigg( \sum_j |\tilde{P}_j \overline{g}|^2 \bigg)^{1/2} dx\\
		&\lesssim \sup_{\norm{g}_{\Psi}\leq 1} \norm{\bigg( \sum_j  |\tilde{P}_jf|^2 \bigg)^{1/2} }_{\Phi} \norm{\bigg( \sum_j |\tilde{P}_j \overline{g}|^2 \bigg)^{1/2}}_{\Psi}\\
		&\leq \norm{\bigg( \sum_j  |\tilde{P}_jf|^2 \bigg)^{1/2} }_{\Phi}.
	\end{split} \]
	In the last inequality, we used \eqref{5LPs}.
\end{proof}
Now, we prove Theorem \ref{M3}. 
\begin{proof}
	We can follow the proof of Theorem 2 in \cite{Bour89} since we established Lemma \ref{LPest}. By the assumption \eqref{M3cond1}, we have
	\[\begin{split}
		K_{\Phi_0}(E_{2^r}) \lesssim \frac{|E_{2^r}|^{1/2}}{\Phi_0^{-1}(2^r)}
	\end{split}. \]
	By Theorem \ref{M2}, there exists a set $ S_r \subset E_{2^r} $ such that $ |S_r| \approx ( \Phi_1^{-1} (2^r))^2 $ and
	\[ \norm{ \sum_{n \in S_r} a_ne^{2\pi i nx} }_{\Phi_1} \lesssim \bigg(\sum_{n \in S_r}|a_n|^{2}\bigg)^{1/2}. \]
	Let $ S = \cup_{r=1}^\infty S_r $ and $ f(x) = \sum_{n \in S} a_n e^{2\pi i n x} $. By \eqref{5LPapprox}, we have
	\[ \norm{\sum_{n \in S} a_ne^{2\pi in x}}_{\Phi_1} \approx \norm{ \bigg( \sum_{j \in \bN} | \tilde{P}_jf|^2 \bigg)^{1/2} }_{\Phi_1}. \]
	Thus, we obtain
	\[ \begin{split}
		\norm{\bigg( \sum_{j \in \bN} |\tilde{P}_j f|^2 \bigg)^{1/2} }_{\Phi_1}^2 &= \inf\bigg\{  k^2 : \int \Phi_1\bigg( \frac{1}{k} \bigg(\sum_{j \in \bN} |\tilde{P}_jf|^2 \bigg)^{1/2} \bigg)dx \leq 1  \bigg\}\\
		&=\inf\bigg\{  s : \int \Phi_1\bigg(  \bigg( \frac{1}{s}\sum_{j \in \bN} |\tilde{P}_jf|^2 \bigg)^{1/2} \bigg)dx \leq 1  \bigg\}.
	\end{split} \]
	By Lemma \ref{5newconv}, there exists $ \tilde{\Phi}_1 $ such that $ {\Phi}_1 \sim \tilde{\Phi}_1$ and $ \tilde{\Phi}_1(u^{1/2}) $ is convex so that the Orlicz norm with respect to $ \tilde{\Phi}_1(u^{1/2}) $ is well defined. Therefore, we have
	\[ \begin{split}
	\norm{\bigg( \sum_{j \in \bN} |\tilde{P}_j f|^2 \bigg)^{1/2} }_{\Phi_1}^2 &		\lesssim \inf\bigg\{  s : \int \tilde{\Phi}_1\bigg(  \bigg( \frac{1}{s}\sum_{j \in \bN} |\tilde{P}_jf|^2 \bigg)^{1/2} \bigg)dx \leq 1  \bigg\}\\
	&=\norm{\sum_{j \in \bN} |\tilde{P}_jf|^2}_{\Gamma} \qquad \mathrm{where}\  
	 \Gamma(u) = \tilde{\Phi}_1(u^{1/2})\\
	&\leq \sum_{j \in \bN} \norm{|\tilde{P}_jf|^2}_{\Gamma}  = \sum_{j \in \bN} \norm{\tilde{P}_jf}_{\tilde{\Phi}_1}^2 \approx \sum_{j \in \bN} \norm{\tilde{P}_jf}_{\Phi_1}^2 \\
	&\lesssim \sum_{j \in \bN} \norm{\tilde{P}_jf }_2^2 \approx \norm{f}_2^2.
	\end{split} \]  
Note that we only used convexity of the function when we used $\Gamma(u)$. Hence, $ S $ is a $ \Lambda(\Phi_1) $-set. Now, let us consider 
\[ k = \frac{( \Phi_1^{-1} (2^r) )^2}{\Phi_2^{-1}(2^r)}. \]
Then, we have  
\begin{equation*}
	\int_{|x| \leq 2^{-r}/10 } \Phi_2 \bigg(\frac{1}{k} {\sum_{n \in S_r} e^{2\pi i n x}} \bigg) \gtrsim \frac{2^{-r}}{10} \Phi_2 \bigg( \frac{1}{2k} ( \Phi_1^{-1}(2^r) )^2 \bigg)  \gtrsim 1.
\end{equation*}
Therefore, it follows that
\[ \norm{\sum_{n \in S_r} e^{2\pi i n x}}_{\Phi_2} \gtrsim \frac{ (\Phi_1^{-1} (2^{r}))^2 }{\Phi_2^{-1}(2^r)} \approx \frac{\Phi_1^{-1}(2^r)}{\Phi_2^{-1}(2^r)} \norm{\sum_{n \in S_r} e^{2\pi i n x}}_2. \]
By using \eqref{M3cond2} and Lemma \ref{5inv_est}, we obtain
\[ \infty = \sup_{ x \geq 1} \frac{\Phi_1^{-1}(x)}{\Phi_2^{-1}(x) } \approx \sup_{ r \geq 1} \frac{\Phi_1^{-1}(2^r)}{\Phi_2^{-1}(2^r) }  \]
Therefore, $ S $ is not a $ \Lambda(\Phi_1) $-set.
\end{proof}

	\end{document}